\documentclass{amsart}

\usepackage[latin1]{inputenc}
\usepackage[T1]{fontenc}
\usepackage{enumerate}
\usepackage{graphicx}
\usepackage{xcolor}

\usepackage[english]{babel}

\usepackage{amsmath,amsfonts,amssymb,amsthm,amscd}

\newtheorem{theorem}{Theorem}
\newtheorem{lemma}[theorem]{Lemma}
\newtheorem{corollary}[theorem]{Corollary}
\newtheorem{proposition}[theorem]{Proposition}

\theoremstyle{definition}

\newtheorem{example}[theorem]{Example}

\newtheorem{question}{Question}

\DeclareMathOperator{\conv}{conv}
\DeclareMathOperator{\Graph}{Graph}
\DeclareMathOperator{\Range}{Range}

\renewcommand{\Re}{\operatorname{Re}}
\renewcommand{\Im}{\operatorname{Im}}

\theoremstyle{remark}
\newtheorem{remark}[theorem]{Remark}

\numberwithin{equation}{section}

\newcommand{\N}{\mathbb N}
\newcommand{\K}{\mathbb K}
\newcommand{\R}{\mathbb R}

\newcommand{\C}{\mathbb C}
\newcommand{\D}{\mathbb D}

\title{Hypercyclic algebras for convolution and composition operators}
\author[J.\ B\`{e}s, J.\ A.\ Conejero and D.\ Papathanasiou]{J.\ B\`{e}s, J.\ A.\ Conejero, and D.\ Papathanasiou}
\address{J. B\`{e}s, Department of Mathematics and Statistics,
Bowling Green State University,
Bowling Green, Ohio 43403,
USA}
\email{jbes@bgsu.edu}
\address{J. A. Conejero, Instituto Universitario de Matem\'atica Pura y Aplicada,
Universitat Polit\`ecnica de Val\`encia, E-46022 Valencia }
\email{aconejero@mat.upv.es}
\address{D. Papathanasiou, Department of Mathematics and Statistics,
Bowling Green State University,
Bowling Green, Ohio 43403,
USA}
\email{dpapath@bgsu.edu}
\thanks{This work is supported in part by MEC, Project
MTM 2016-7963-P. 
We also thank \'Angeles Prieto and Jes\'us A.\ Jaramillo for comments and suggestions}

\date{February 2, 2018}

\allowdisplaybreaks[3]
\begin{document}
\begin{abstract}
We provide an alternative proof to those by Shkarin and by Bayart and Matheron that the operator $D$ of complex differentiation supports a hypercyclic algebra on the space of entire functions. In particular we obtain hypercyclic algebras for many convolution operators not induced by polynomials, such as $\mbox{cos}(D)$, $De^D$, or $e^D-aI$, where $0<a\le 1$.
In contrast, weighted composition operators on function algebras of analytic functions on a plane domain fail to support supercyclic algebras.
\end{abstract}

\dedicatory{Dedicated to Professor Jes\'us A. Jaramillo on the occasion
	of his $60$th birthday.}

\maketitle
{\large 
\section{Introduction}
A special task in linear dynamics is to understand the algebraic and topological properties of the set
\[
HC(T)=\{ f\in X: \ \ \{ f, Tf, T^2f,\dots \} \mbox{ is dense in $X$} \}
\]
of hypercyclic vectors for a given operator $T$ on a topological vector space $X$.    It is well known that in general $HC(T)$ is always connected and is either empty or contains a dense infinite-dimensional linear subspace (but the origin), see \cite{wengenroth2003hypercyclic}. Moreover, when $HC(T)$ is non-empty it sometimes contains (but zero) a closed and infinite dimensional linear subspace, 
but not always
\cite{bernal-gonzalez_montes-rodriguez1995non-finite-dimensional,gonzalez_leon-saavedra_montes-rodriguez2000semi-fredholm}; see also \cite[Ch.\ 8]{bayart_matheron2009dynamics} and \cite[Ch.\ 10]{grosse-erdmann_peris-manguillot2011linear}.

When $X$ is a topological algebra it is natural to ask whether $HC(T)$ can contain, or must always contain, a subalgebra (but zero) whenever it is non-empty; any such subalgebra is said to be a {\em hypercyclic algebra } for the operator $T$.  Both questions have been answered by considering convolution operators on the space $X=H(\mathbb{C})$ of entire functions on the complex plane $\mathbb{C}$, endowed with the compact-open topology; that convolution operators (other than scalar multiples of the identity) are hypercyclic was established by Godefroy and Shapiro \cite{godefroy_shapiro1991operators}, see also \cite{birkhoff1929demonstration,maclane1952sequences,aron_markose2004on}, together with the fact that convolution operators on $H(\C)$ are precisely those of the form
\[
f\overset{\Phi(D)}{\mapsto} \sum_{n=0}^\infty a_n D^nf   \ \ \ (f\in H(\C))
\]  
where $\Phi(z)=\sum_{n=0}^\infty a_n z^n \in H(\C)$ is of (growth order one and finite) exponential type (i.e., $|a_n|\le M\ \frac{R^n}{n!} (n=0, 1,\dots )$, for some $M, R>0$) and where $D$ is the operator of complex differentiation.
Aron et al \cite{aron_conejero_peris_seoane-sepulveda2007powers,aron_conejero_peris_seoane-sepulveda2007sums} showed that no translation operator $\tau_a$ on $H(\C)$
\[
\tau_a (f)(z)=f(z+a) \ \ f\in H(\mathbb{C}), z\in\mathbb{C}
\]
can support a hypercyclic algebra, in a strong way:

\begin{theorem} \label{T:-1} {\rm {\bf  (Aron, Conejero, Peris, Seoane)}}
For each integer $p>1$ and each $f\in HC(\tau_a)$,  the non-constant elements  of the orbit of $f^p$ under $\tau_a$ are those entire functions for which the multiplicities of their zeros are integer multiples of $p$.   
\end{theorem}

 In sharp contrast with the translations operators,  they also  showed  that the collection of entire functions $f$ for which every power $f^n$ $(n=1,2,\dots )$ is hypercyclic for $D$ 
is residual in $H(\mathbb{C})$. 
Later Shkarin \cite[Thm.\ 4.1]{shkarin2010on} showed that $HC(D)$ contained both a hypercyclic subspace and a hypercyclic algebra, and with a different approach Bayart and Matheron
\cite[Thm.\ 8.26]{bayart_matheron2009dynamics} also showed that the set of $f\in H(\mathbb{C})$ that  generate an algebra consisting entirely (but the origin) of hypercyclic vectors for $D$ is residual in $H(\mathbb{C})$, and by using the latter approach 
we now know the following:

\begin{theorem} \label{T:0}  {\rm {\bf (Shkarin \cite{shkarin2010on}, Bayart and Matheron \cite{bayart_matheron2009dynamics}, B\`es, Conejero, Papathanasiou \cite{bes_conejero_papathanasiou2016convolution})}}
 Let $P$ be a non-constant polynomial with $P(0)=0$. 
Then the set of functions $f\in H(\C)$ that generate a hypercyclic algebra for $P(D)$ is residual in $H(\C)$. 
\end{theorem}

Motivated by the above results we consider the following question.
\begin{question} \label{Q:2}
{\em Let $\Phi\in H(\C )$ be of exponential type so that the convolution operator $\Phi(D)$ supports a hypercyclic algebra. Must $\Phi$ be a polynomial?  Must $\Phi(0)=0$?}
\end{question}

In Section~\ref{S:T1} we answer both parts of Question~\ref{Q:2} in the negative, by establishing for example that $\Phi(D)$ supports a hypercyclic algebra when $\Phi(z)=\mbox{cos}(z)$ and when $\Phi(z)=z e^z$ (Example~\ref{E:cos(z)} and Example~\ref{E:ze^z}), as well as when $\Phi(z)=(a_0+a_1 z^n)^k$ with $|a_0|\le 1$ and $0 \ne a_1$ and when $\Phi(z)=e^z-a$ with $0<a\le 1$ (Corollary~\ref{C:2} and Example~\ref{E:e^z-a}). All such examples 
are derived from our main result:

 \begin{theorem}   \label{T:1}  
Let $\Phi\in H(\C )$ be 
 of finite exponential  type so that the level set $\{ z\in\mathbb{C}: \ |\Phi(z)|=1 \}$ contains a non-trivial, strictly convex compact arc $\Gamma_1$ satisfying
\begin{equation} \label{eq:T1.1}
\conv(\Gamma_1\cup \{ 0 \} ) \setminus ( \Gamma_1\cup \{ 0 \} ) \subseteq \Phi^{-1} (\mathbb{D}).
\end{equation}
Then the set of entire functions
that generate a hypercyclic algebra for the convolution operator $\Phi(D)$ is residual in $H(\C)$. 
\end{theorem}
Here for any  $A\subset \C$ the symbol $\conv(A)$ denotes its convex hull, and
$\D$ denotes the open unit disc.
Also, an arc $\mathcal{C}$ is said to be strictly convex provided for each $z_1, z_2$ in $\mathcal{C}$ the segment $\conv(\{ z_1, z_2\})$ intersects $\mathcal{C}$ at at most two points.
\vspace{.1in}

In Section~\ref{S:CompositionOperators} we consider the following question, motivated by Theorem~\ref{T:-1}:

\begin{question} \label{Q:1}
{\em Can a multiplicative operator on a $F$-algebra support a hypercyclic algebra? In particular, can a composition operator support a hypercyclic algebra on some space $H(\Omega)$ of holomorphic functions on a planar domain $\Omega$? }
\end{question}
The study of hypercyclic composition operators on spaces of holomorphic functions may be traced back to the classical examples by Birkhoff~\cite{birkhoff1929demonstration} and by Seidel and Walsh~\cite{SeWa41}, and is described in a recent survey article by Colonna and Mart\'{i}nez-Avenda\~{n}o \cite{ColAve2017}. 
Grosse-Erdmann and Mortini showed that the space $H(\Omega)$ of holomorphic functions on a  planar domain $\Omega$ supports a hypercyclic composition operator if and only if $\Omega$ is either simply connected or infinitely connected \cite{GroMor09}.

 We show in Section~\ref{S:CompositionOperators} that a given multiplicative operator $T$ on an $F$-algebra $X$ supports a hypercyclic algebra if and only if $T$ is hypercyclic and for each non-constant polynomial $P$ vanishing at zero the map $X\to X$, $f\mapsto  P(f)$ has dense range (Theorem~\ref{P:duo}). We use this to derive that for each $0\ne a\in\R$ the translation operator $\tau_a$ supports a hypercyclic algebra on $C^\infty(\R, \C)$ (Corollary~\ref{C:HcAsmoothtranslationsC}) but fails to support a hypercyclic algebra on $C^\infty(\R, \R)$ (Corollary~\ref{C:T_aR}).  Here by $C^\infty(\R, \K)$ we denote the Fr\'{e}chet space of $\K$-valued infinitely differentiable functions on $\R$ whose topology is given by the seminorms
\[
P_k(f)=\mbox{max}_{0\le j\le k} \mbox{max}_{t\in [-k, k]} |f^{(j)}(t) | \ \ \ \ (f\in C^\infty(\R,\K),\ k\in\N).
\] 
Finally, we show that no weighted composition operator $C_{\omega, \varphi}:H(\Omega)\to H(\Omega)$, $f\mapsto \omega (f\circ \varphi)$, supports a supercyclic algebra (Theorem~\ref{T:trio}). 
Recall that a vector $f$ in an $F$-algebra $X$ is said to be {\em 
supercyclic} for a given operator $T:X\to X$ provided 
\[
\C\cdot \mbox{Orb}(f, T) =\{ \lambda T^nf: \ \lambda\in\C, n=0,1,\dots \}
\]
is dense in $X$. 
Accordingly, any subalgebra of $X$ consisting entirely (but zero) of supercyclic vectors for $T$ is said to be a {\em supercyclic algebra}.

 We refer to \cite{grosse-erdmann_peris-manguillot2011linear, bayart_matheron2009dynamics} for a broad view of the field of linear dynamics and to  \cite{aron_bernal-gonzalez_pellegrino_seoane-sepulveda2015lineability,bernal-gonz\'alez_pellegrino_seoane} for an overview of the search of algebraic structures in non-linear settings.

\section{Proof of Theorem~\ref{T:1} and its consequences} \label{S:T1}

The proofs of Theorem~\ref{T:0} and of its earlier versions
exploit the shift-like behaviour of the operator $D$ on $H(\C )$ \cite{shkarin2010on, bayart_matheron2009dynamics, bes_conejero_papathanasiou2016convolution}. 
Our approach for Theorem~\ref{T:1} exploits instead the rich source of eigenfunctions that convolution operators on $H(\C)$ have (i.e.,
\[
\Phi(D) (e^{\lambda z})=\Phi (\lambda) e^{\lambda z} 
\]
for each $\lambda\in \C$ and each $\Phi\in H(\C)$ of exponential type) as well as the following key result by Bayart and Matheron:

\begin{proposition} \label{C:111}   {\bf {\rm (Bayart-Matheron~\cite[Remark~8.26]{bayart_matheron2009dynamics}})}
Let $T$ be an operator on a separable $F$-algebra $X$ so that for each triple $(U, V, W)$ of non-empty open subsets of $X$ with $0\in W$ and for each $m\in\N$ there exists $P\in U$ and $q\in\N$ so that
\begin{equation}  \label{eq:0*}
\begin{cases}
T^q(P^j) \in W \ \ \ \mbox{ for } 0\le j < m,\\
T^q (P^m)\in V.
\end{cases}
\end{equation}
Then the set of  elements of $X$ that generate a hypercyclic algebra for $T$ is residual in $X$.
\end{proposition}

We start by noting the following invariant for composition operators with homothety symbol.

\begin{lemma}  \label{L:0} 
Let $\Phi \in H(\C )$ be of exponential type, and let $\varphi :\C \to \C$, $\varphi (z)=a z$ be a homothety on the plane, where $0\ne a\in\C$. Then  $\Phi_a :=C_\varphi (\Phi)$ is of exponential type and 
\[
C_\varphi (HC(\Phi_a (D)))=HC (\Phi (D)).
\]
In particular, the algebra isomorphism $C_\varphi :H(\C )\to H(\C )$ maps hypercyclic algebras of $\Phi_a (D)$ onto hypercyclic algebras of $\Phi (D)$.
\end{lemma}
\begin{proof}
For each $f\in H(\C)$ we have  $C_\varphi (f)(z)=f( a z)$ $(z\in\C)$, and thus 
\[
D^k C_\varphi (f) = a^k C_\varphi D^k (f)      \ \  (k=0,1,2,\dots ).
\]
Hence given $\Phi (z)=\sum_{k=0}^\infty c_k z^k$ of exponential type  $\Phi_a :=C_\varphi (\Phi)$ is clearly of exponential type and
\[
\begin{aligned}
\Phi (D) C_\varphi (f) &= \sum_{k=0}^\infty c_k D^k C_\varphi (f) = \sum_{k=0}^\infty c_k a^k C_\varphi D^k (f)\\ & = C_\varphi \left(   \sum_{k=0}^\infty c_k a^k D^k \right) (f) \\
& = C_\varphi \Phi_a (D) (f) \ \ \ \ (f\in H(\C )).
\end{aligned}
\]
So $\Phi_a(D)$ is conjugate to $\Phi(D)$ via the algebra isomorphism $C_\varphi$. 
\end{proof}
 
\begin{remark} \label{R:2} \
\begin{enumerate}
\item \  Lemma~\ref{L:0} is a particular case of the following Comparison Principle for Hypercyclic Algebras. {\em Any operator $T:X\to X$ on a  Fr\'{e}chet algebra $X$ that is quasi-conjugate via a multiplicative operator $Q:Y\to X$ to an operator $S:Y\to Y$ supporting a hypercyclic algebra must also support a hypercyclic algebra. Indeed, if $A$ is a hypercyclic algebra for $S$, then $Q(A)=\{ Qy: y\in A \}$ is a hypercyclic algebra for $T$.} 

\item \  
If $\Phi \in H(\C )$ satisfies the assumptions of Theorem~\ref{T:1}, then so will $\Phi_a:=C_\varphi(\Phi)$ for any non-trivial homothety $\varphi(z)=az$.  Indeed, notice that for any $r>0$ we have
\[
a \Phi_a^{-1} (r\partial\D)=\Phi^{-1}(r\partial\D).
\]
Hence if $\Gamma \subset \Phi^{-1}(r\partial\D)$  is a smooth arc satisfying
\[
\conv(\Gamma\cup\{ 0 \} ) \setminus (\Gamma \cup \{ 0 \}) \subset \Phi^{-1} (r\D),
\]
then $\Gamma_a:= \frac{1}{a} \Gamma \subset \Phi_a^{-1} (r\partial\D)$ is a smooth arc satisfying
\[
\conv(\Gamma_a\cup\{ 0 \} ) \setminus (\Gamma_a \cup \{ 0 \}) \subset \Phi_a^{-1} (r\D).
\]
Moreover, if $\Gamma$ is a strictly convex, compact, simple and non-closed arc whose convex hull does not contain the origin, say, then $\Gamma_a$ will share each corresponding property as  these are invariant under homothecies. In particular, the angle difference between the endpoints of $\Gamma$ is the same as the corresponding quantity in $\Gamma_a$.
\end{enumerate}
\end{remark}

The next result ellaborates on the geometric assumption of Theorem~\ref{T:1}. Here for any $0\ne z\in\C$ we denote by $\arg(z)$ the argument of $z$ that belongs to  $[0, 2\pi)$.

\begin{proposition} \label{P:R2}
 Let $\Phi \in H(\C )$ and let $\Gamma \subset \Phi^{-1}(r \partial\D )$ be a simple, strictly convex arc with endpoints $z_1$, $z_2$ satisfying $0<\arg(z_1)<\arg(z_2)<\pi$ and $\Re(z_1)\ne \Re(z_2)$, where $r>0$. Suppose that
 $0\notin\conv(\Gamma)$ and that
\begin{equation}\label{eq:P:R2.0}
\Omega := \conv(\Gamma \cup \{ 0 \} ) \setminus (\Gamma \cup \{ 0 \}) \subset \Phi^{-1} (r \D ).
\end{equation}
Then $S(0, z_1, z_2)\setminus \Gamma$ consists of two connected components of which $\Omega$ is the bounded one, where
\[
S(0, z_1, z_2) = \{ 0\ne w\in \C: \ \arg(z_1)\le \arg(w) \le \arg(z_2) \}.
\]
Moreover, 
\[
\  \Omega = \{ tz:\ (t, z)\in (0,1)\times \Gamma \} = \{ tz:\ (t, z)\in (0,1) \times \conv(\Gamma) \}, 
\]
and $
\partial \Omega = [0, z_1)\cup (0, z_2) \cup \Gamma$. In addition, 
\[
\Gamma\cap (I \times (0,\infty))=\Graph(f)\cup \{ z_1, z_2\}
\]
for some smooth function $f:I\to\R$, where $I$ is the closed interval with endpoints $\Re(z_1)$ and $ \Re(z_2)$ and where $f$ is concave up if $\Re(z_1)<\Re(z_2)$ and concave down if $\Re(z_2)<\Re(z_1)$.
\end{proposition}

{In Figure \ref{fig:proposition7} we illustrate one case of the statement of this Proposition~\ref{P:R2}.}

\begin{figure}[ht] 
\centering
\includegraphics[width=80mm]{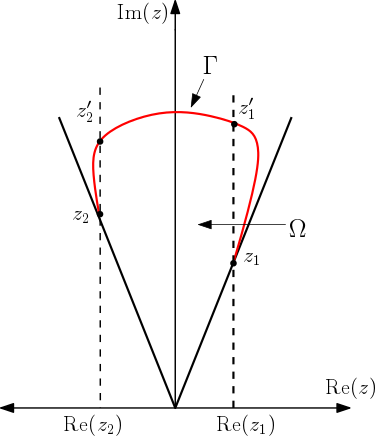}
\caption{A representation of the sets appearing in Proposition \ref{P:R2}.}
\label{fig:proposition7}
\end{figure}

\begin{proof}
Since $|\Phi |\le r$ on $\conv(\Gamma\cup\{ 0\})$ by $\eqref{eq:P:R2.0}$, the maximum modulus principle ensures that 
\begin{equation} \label{eq:R1}
\Gamma \cap \mbox{int}( \conv(\Gamma\cup \{ 0 \} )) =\emptyset.
\end{equation} 
We claim that
\begin{equation} \label{eq:Rclaim}
\Gamma \subset \{ 0 \ne w\in\C: \ \arg(w)\in [\arg(z_1), \arg(z_2)] \}.
\end{equation}
To see this, notice that since $0\notin \conv(\Gamma)$ the arc $\Gamma$ cannot  intersect the ray $\{ t e^{i (\arg(z_2)+\pi)}: t\ge 0 \}$, and by $\eqref{eq:R1}$ it cannot intersect the interior of the triangle $\conv\{ 0, z_1, z_2\}$, either. Also, notice that if $H$ denotes the open half-plane  not containing $0$ and with boundary
\[
\partial H = \{ z_1+t(z_2-z_1):\ t\in\R \},
\]
then 
\begin{equation} \label{eq:R2}
\emptyset = \Gamma\cap H\cap \{0\ne w\in\C: \ \arg(w)<\arg(z_1) \},
\end{equation}
as any $z\in \Gamma \cap H$ with $\arg(z)<\arg(z_1)$ would make $z_1\in \mbox{int}(\conv(\{ z, z_2, 0\}))$, contradicting $\eqref{eq:R1}$. Finally, since $\Gamma$ is simple it now follows from $\eqref{eq:R2}$ that
\[
\emptyset = \Gamma \cap \{ 0\ne w\in \C: \ \arg(w)\in [\pi +\arg(z_2), 2\pi)\cup [0, \arg(z_1)) \},
\]
and thus any $w\in\Gamma$ satisfies $\arg(z_1) \le \arg(w)$. By a similar argument, each $w\in\Gamma$ satisfies  $\arg(w)\le \arg(z_2)$, and $\eqref{eq:Rclaim}$ holds.
Next, using $\eqref{eq:R1}$ and the continuity of the argument on $S(0, z_1, z_2)$ it is simple now to see that for each $\theta\in [\arg(z_1), \arg(z_2)]$ the ray
\[
\{ te^{i\theta} : \ t\ge 0 \}
\]
intersects $\Gamma$ at exactly one point, giving the desired description for $\Omega$. For the final statement, assume $\Re(z_2)< \Re(z_1)$ (the case $\Re(z_1)<\Re(z_2)$ follows with a similar argument).

Notice that for each $x=t \Re(z_2)+(1-t)\Re(z_1)$ with $0<t<1$ there exists a unique $y\in\R$
so that
\begin{equation}  \label{eq:P:R2:2}
(x,y)\in \Gamma \ \mbox{ with } \ y\in [t\Im(z_2)+(1-t)\Im(z_1), \infty).
\end{equation}
Indeed, the  continuous path $\Gamma$ from $z_1$ to $z_2$ lies in $S(0, z_1, z_2)$ and only meets the closed triangle $\conv(\{ 0, z_1, z_2 \})$ at $z_1$ and $z_2$, so the existence of a $y$ verifying $\eqref{eq:P:R2:2}$ follows (it also follows for the cases $t=0,1$, in which case there may exist up to two values per  endpoint, by $\eqref{eq:Rclaim}$).  To see the uniqueness, if $y_2>y_1 > t\Im(z_2)+(1-t)\Im(z_1)$ with $(x, y_1), (x,y_2)\in\Gamma$, then
\[
(x,y_1)\in\mbox{int}(\conv(\{ z_1, z_2, x+iy_2\})\cap \Gamma \subset \Omega \cap \Gamma =\emptyset,
\]
a contradiction.  Hence $\eqref{eq:P:R2:2}$ defines a smooth function $f:[\Re(z_1), \Re(z_2)]\to (0,\infty)$ whose graph $\Gamma_0$
 is a subarc of $\Gamma$, provided that if at either endpoint $x\in\{ \Re(z_1), \Re(z_2)\}$ there are two values $y$ satisfying $x+iy\in\Gamma$ we let $f(x)$ be the largest of such two values. 
\end{proof}

Finally, Lemma~\ref{L:1} below will enable us to apply Proposition~\ref{C:111}.
Recall that for  a planar smooth curve $\mathcal{C}$ with parametrization $\gamma : [0,1]\to \C$, $\gamma (t)=x(t)+i y(t)$, its signed curvature at a point $P=\gamma (t_0)\in \mathcal{C}$ is given by
\[
\kappa (P):= \frac{  x'(t_0) y''(t_0) - y'(t_0) x''(t_0) }{ |\gamma'(t_0)|^3 }.
\]
and its unsigned curvature at $P$ is given by $|\kappa (P)|$.
It is well-known that $|\kappa (P)|$ does not depend on the parametrization selected, and that the signed curvature $\kappa (P)$ depends only on the choice of orientation selected for $\mathcal{C}$.  It is simple to see that any straight line segment has zero curvature. We say that $\mathcal{C}$ is {\em strictly convex} provided each segment with endpoints in the arc only intersects the arc at these points. 
Notice also that for the particular case when $\mathcal{C}$ is given by the graph of a function $y=f(x)$, $a\le x\le b$, (and oriented from left to right), its signed curvature at a point $P=(x_0, f(x_0))$ is given by
\[
\kappa (P)= \frac{ f''(x_0)}{(1+(f'(x_0))^2 )^{\frac{3}{2}}}.
\]
In particular, $\kappa < 0$ on $\mathcal{C}$ if and only if $y=f(x)$ is concave down (i.e.,  $ (1-s) f(a_1) + s f(b_1)<  f((1-s)a_1 + s b_1) $     for  any $s\in (0, 1)$ and any subinterval $[a_1, b_1]$ of $[a,b]$).

\begin{lemma} \label{L:1} 
Let $\Phi\in H(\C )$ be 
 of exponential  type supporting a non-trivial, strictly convex compact arc $\Gamma_1$ contained in
 $\Phi^{-1} (\partial \D)$ so that
\[ 
\conv(\Gamma_1\cup \{ 0 \} ) \setminus ( \Gamma_1\cup \{ 0 \} ) \subseteq \Phi^{-1} (\mathbb{D}).
\] 
 Then for each 
 $m\in \N$ there
 exist $r>1$, a non-trivial, strictly convex smooth arc $\Gamma \subset \Phi^{-1} ( r \partial \D )\cap \{ tz:  \ (t,z)\in (0,\infty)\times \Gamma_1 \}$ and $\epsilon >0$ so that   
\begin{equation} 
\label{eq:L1.1}
\conv(\Gamma \cup \{ 0 \} ) \setminus  \Gamma  \subseteq \Phi^{-1} (r \mathbb{D}).
\end{equation}
and
\begin{equation}
\label{eq:L1.2}
\Lambda + \sum_{k=1}^j  \frac{1}{m} \Gamma \subset \Omega \ \ \mbox{ and } \ \ \sum_{k=1}^j  \frac{1}{m} \Gamma \subset\Omega  \ \mbox{ for each $1 \le j <m$,}
\end{equation}
where
\[
\begin{aligned}
\Omega &:= \conv(\Gamma_1\cup \{ 0 \} ) \setminus ( \Gamma_1\cup \{ 0 \} ) \\
\Lambda &:= \Omega \cap D(0, \epsilon)\cap  \conv(\Gamma\cup \{ 0 \}).  
\end{aligned}
\]
\end{lemma}
{In Figure \ref{fig:lemma8} we illustrate the different sets appearing in the statement of  Lemma~\ref{L:1}.}

\begin{figure}[ht] 
\centering
\includegraphics[width=80mm]{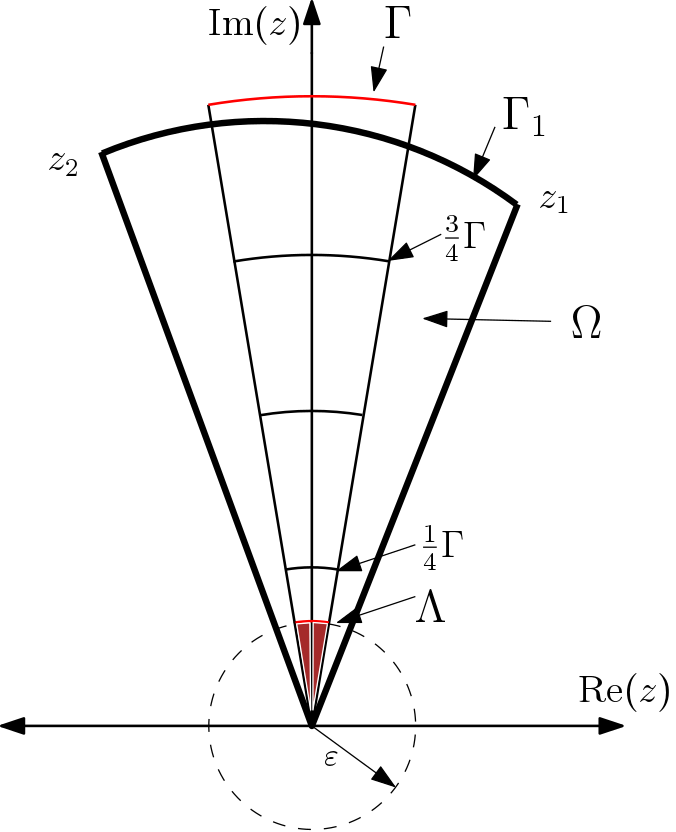}
\caption{The sets appearing in Lemma \ref{L:1}, case $m=4$.}
\label{fig:lemma8}
\end{figure}

\begin{proof}

Since $\Gamma_1$ is strictly convex, replacing it by a subarc if necessary we may further assume by Remark~\ref{R:2}.(2) that $\Gamma_1$ is simple and with endpoints $z_1, z_2$ satisfying 
$0<\arg(z_1)<\arg(z_2)<\pi$ and $\Re(z_2)<\Re(z_1)$ and so that $0\notin\conv(\Gamma_1)$.
By Proposition~\ref{P:R2},
\begin{equation} \label{eq:L1-1}
\Omega =\{ tz: \ (t,z)\in (0,1)\times \conv(\Gamma_1) \} \subset S(0, z_1, z_2),
\end{equation}
with $
\partial\Omega = [0, z_1)\cup \Gamma_1 \cup (0, z_2)$ and we may assume $\Gamma_1$ is the graph of a concave down function $f:[\Re(z_2), \Re(z_1)]\to (0, \infty)$ (i.e., replacing $z_j$ by $z_j'=\Re(z_j)+if(\Re(z_j))$, $j=1,2$, if necessary).
Now, pick $z_0\in \Gamma_1\setminus \{ z_1, z_2 \}$ with $\Phi' (z_0)\ne 0$, and let $w_0:=\Phi(z_0)=e^{i\theta_0}$, where $\theta_0\in [0, 2\pi)$. Choose $\rho>0$ small enough so that the only solution to
\[
\Phi(z)=w_0
\]
in $D(z_0, \rho)$ is at $z=z_0$, and so that $D(z_0, \rho )\cap ([0, z_1] \cup [0, z_2])=\emptyset$.  Next, pick
\[
0<s<\mbox{min}\{ |\Phi(z)-w_0|: \ |z-z_0|=\rho \}
\]
and let $0<\delta < \mbox{min}\{ 1, s\}$ so that the polar rectangle
\[
R_\delta := \{ z=r e^{i\theta }: \ (r,\theta)\in [1-\delta, 1+\delta]\times [\theta_0-\delta, \theta_0+\delta] \}
\]
is contained in $D(w_0, s)$. Then 
\[
g:R_\delta \to D(z_0, \rho), \ g(w)= \frac{1}{2\pi i} \underset{|z-z_0|=\rho}{\int} \frac{z \Phi'(z)}{\Phi (z)-w } dz
\]
defines a univalent holomorphic function satisfying that
\begin{equation} \label{eq:2}
\Phi \circ g = \mbox{identity on $R_\delta$,}
\end{equation}
see e.\ g.\ \cite[p.\ 283]{gamelin}.
So $W:= g(R_\delta)$ is a connected compact neighborhood of $z_0$, and $\Phi$ maps $W$ biholomorphically onto $R_\delta$. Hence for each $1-\delta \le r \le 1+\delta $
\[
\eta_r := g(R_\delta \cap r\partial \D)
\]
is a smooth arc contained in $W\cap \Phi^{-1}(r\partial\D)$. In particular, $\eta_1=W\cap\Gamma_1$ is a strictly convex subarc of $\Gamma_1$. 
Next, notice that since 
\[
W\cap \Omega \ \ \mbox{ and } \ \ W\cap \mbox{Ext}(\Omega ) 
\]
are the two connected components of $g(R_\delta\setminus \partial\D )= W\setminus \eta_1$
and $\Omega \subseteq \Phi^{-1}(\D)$, by $\eqref{eq:2}$  the homeomorphism $g:R_\delta\setminus \partial\D \to W\setminus \eta_1   $ must satisfy
\[
\begin{aligned}
g(R_\delta \cap \mbox{Ext}(\D)) &= W\cap \mbox{Ext}(\Omega) \\
g(R_\delta \cap \D)&= W\cap \Omega.
\end{aligned}
\]
Hence 
\[
W\cap \overline{\mbox{Ext}(\Omega)} = \underset{1\le r \le 1+\delta}{\cup} \eta_r
\]
and $g$ induces a smooth homotopy among the curves $\{ \eta_r \}_{1\le r\le 1+\delta}$. Namely,
each $\eta_r$ $(1\le r\le 1-\delta)$ has the Cartesian parametrization 
\[
\eta_r: \ \begin{cases}   X(r,t) \\
Y(r,t)
\end{cases}  
\ \ \ \theta_0-\delta \le t \le \theta_0 +\delta,
\]
where $X, Y:[1-\delta, 1+\delta]\times [\theta_0-\delta, \theta_0+\delta]\to \R$ are given by
\[
\begin{aligned}
X(r,t)&:= \Re(g)(r e^{it}) \\
Y(r,t)&:= \Im(g)(r e^{it}).
\end{aligned}
\]
Now, for any point $P=g(re^{i\theta})$ in $W$ the (signed) curvature $\kappa^{\eta_r}(P)$ of $\eta_r$ at $P$  is given by 
\[
\kappa^{\eta_r}(P) = \frac{        \frac{   \partial X }{\partial t      } (r,\theta)  \frac{   \partial^2 Y }{\partial^2 t      } (r,\theta) - \frac{   \partial Y }{\partial t      } (r,\theta)  \frac{   \partial^2 X }{\partial^2 t      } (r,\theta)  }{     \left( ( \frac{\partial X}{\partial t}(r,\theta))^2 +    ( \frac{\partial Y}{\partial t}(r, \theta))^2 \right)^{\frac{3}{2}} }.
\]
Hence the map $K : W\to \R$, $K (g(re^{it})) := \kappa^{\eta_r}(P)$, is continuous.  Now, since $\eta_1$ is strictly convex there exists some $P=g(e^{i\theta_1})$ in $\eta_1$ for which each of $\kappa^{\eta_1}(P)$ and $\frac{\partial X}{\partial t}(1,\theta_1)$ is non-zero. Hence by the continuity of $K$ and of $\frac{\partial X}{\partial t}$ we may find some $0<\delta'<\delta$ so that the polar rectangle 
\[
R_{\delta'} := \{ z=r e^{i\theta }: \ (r,\theta)\in [1-\delta', 1+\delta']\times [\theta_1-\delta', \theta_1+\delta'] \}
\]
is contained in the interior of $R_\delta$ and so that $K$  and $\frac{\partial X}{\partial t}$  are bounded away from zero on $g(R_{\delta'})$ and on $R_{\delta'}$, respectively. 

In particular, either  $\frac{\partial X}{\partial t}>0$ \ or  $\frac{\partial X}{\partial t}<0$ on $R_{\delta'}$, and 
either $K>0$ or $K<0$ on $g(R_{\delta'})$. So each $\eta_r\cap g(R_{\delta'})$ $(1\le r < 1+\delta')$ is the graph of a smooth function 
\[
f_r:(a_r, b_r)\to (0, \infty),
\]
with \[
(a_r, b_r)=\begin{cases} (X(r, \theta_1-\delta'), X(r, \theta_1+\delta')) &\mbox{ if $\frac{\partial X}{\partial t}>0$ on $R_{\delta'}$} \\
(X(r, \theta_1+\delta'), X(r, \theta_1-\delta')) &\mbox{ if $\frac{\partial X}{\partial t}<0$ on $R_{\delta'}$.} 
\end{cases}
\]
Since $g(re^{it})\underset{r\to 1}{\to} g(e^{it})$ uniformly on $t\in [\theta_1-\delta, \theta_1+\delta]$, 
 so \[(a_r, b_r)\underset{r\to1}{\to}(a_1, b_1)\] and fixing a non-trivial compact subinterval $[a,b]$ of $(a_1, b_1)$  there exists 
there exists $0<\delta''<\delta'$ so that
\[
[a,b]\subset \cap_{1\le r\le 1+\delta''} (a_r, b_r).
\]
So for each $1<r\le 1+\delta''$
\[
\eta_r'=\{ (x, f_r(x)); \ x\in [a,b] \}
\]
is a subarc of $\eta_r$. In particular, $f_1=f$ on $[a, b]$ must be a concave down function, and so must be each  $f_r$ with $1\le r\le 1+\delta''$. Thus choosing $r>1$ close enough to $1$ the arc $\Gamma:=\eta_r'$
satisfies
\[
\conv(\Gamma\cup\{ 0\})\setminus (\Gamma\cup\{ 0\}) \subset \Phi^{-1}(r\D )\cap \{ tz:  \ (t,z)\in (0,\infty)\times \Gamma_1 \}   
\]
and
\[
\sum_{k=1}^j \frac{1}{m}\Gamma \subset \Omega \ \ \mbox{for $j=1,\dots ,m-1$.}
\]
By the compactness of $\Gamma$ we may now get $\epsilon >0$ small enough so that the subsector 
\[
\Lambda := \Omega \cap D(0, \epsilon)\cap  \conv(\Gamma\cup\{ 0\})
\]
satisfies that
\[
\Lambda + \sum_{k=1}^j \frac{1}{m}\Gamma \subset \Omega  \ \ \mbox{for $j=1,\dots ,m-1$,}
\] and Lemma~\ref{L:1} holds.
\end{proof}

We are ready now to prove the main result.

\begin{proof}[Proof of Theorem~\ref{T:1}]

Let $U, V$ and $W$ be non-empty open subsets of $H(\C )$, with $0\in W$, and let $1\le m\in \N$ be fixed.
By Proposition~\ref{C:111}, it suffices to find some $f\in U$ and $q\in \N$ so that
\begin{equation}  \label{eq:1}
\begin{aligned}
\Phi(D)^q(f^j)&\in W \ \ \ \mbox{ for $j=1,\dots ,m-1$, }\\
\Phi(D)^q(f^m)&\in V.
\end{aligned}
\end{equation}
The case $m=1$ is immediate as $\Phi(D)$ is topologically transitive, so we may assume $1<m$.
Now, let $r>1$, let $\Gamma \subset \Phi^{-1} ( r \partial \D )$ and let $\Omega$ and the subsector $\Lambda$ be given by Lemma~\ref{L:1}. Since the arc $\Gamma$ is non-trivial and $\Lambda$ has non-empty interior, each of $\Gamma$ and $\Lambda$ has accumulation points in $\C$. Hence
there exist  $(a_k, b_k, \lambda_k, \gamma_k)\in \C\times\C\times \Lambda\times \Gamma$ $(k=1,\dots, p)$ so that
\[
(A, B):=\left( \sum_{k=1}^p a_k e^\frac{\lambda_k z}{m}, \sum_{k=1}^p b_k e^{\gamma_k z}\right) \in U\times V.
\]
 Next, set $R=R_q=\sum_{k=1}^p c_k e^\frac{\gamma_k z}{m}$ , where for each $1\le k \le p$ the scalar $c_k=c_k(q)$ is some solution of 
\[
z^m-\frac{b_k}{( \Phi(\gamma_k) )^q} =0.
\]
Notice that for any $k=1,\dots, p$ we have $|\Phi (\gamma_k)|=r>1$ and thus $|c_k|^m=  \frac{|b_k|}{|\Phi (\gamma_k )|^q } \underset{q\to\infty}{\to } 0$. So 
\begin{equation}\label{eq:a}
R=R_q\underset{q\to\infty}{\to } 0.
\end{equation}
For $1\le j\le m$ we have
\[
(A+R)^j = \sum_{\ell =( u, v)\in \mathcal{L}_j} {j \choose{\ell}} \  a^u \ c^v \ e^{(\frac{u\cdot \lambda + v \cdot \gamma}{m} ) z}
\]
where $\mathcal{L}_j$ conists of those multiindexes $\ell = (u, v)\in \N_0^p\times \N_0^p$ satisfying $\ |\ell|:=|u|+|v|=\sum_{k=1}^p u_k + \sum_{k=1}^p v_k = j$ and where for each $\ell = (u, v)\in \mathcal{L}_j$
\[
\begin{aligned}
a^u&:= a_1^{u_1} \ a_2^{u_2} \cdots a_p^{u_p}, \\
c^v&:= c_1^{v_1} \ c_2^{v_2} \cdots c_p^{v_p}, \mbox{ and }\\
{j\choose\ell}&=\frac{j!}{u_1!\cdots u_p! v_1! \cdots v_p!}.
\end{aligned}
\]
So for $1\le j \le m$ we have 
\[
\Phi(D)^q ((A+R)^j )= \sum_{\ell \in \mathcal{L}_j} U_{j, \ell},
\]
where 
\[ 
\begin{aligned}
U_{j, \ell} &= {j \choose{\ell}} \  a^u \ c^v \ 
\left( \Phi(\frac{u\cdot \lambda + v\cdot \gamma}{m}) \right)^q
e^{(\frac{u\cdot \lambda + v \cdot \gamma}{m} ) z}  \\
&=  {j \choose{\ell}} \  a^u \ b^{\frac{v}{m}} \ 
\left( 
\frac{
\Phi(\frac{u\cdot \lambda + v\cdot \gamma}{m})
 }{    \prod_{k=1}^p  \Phi(\gamma_k )^{\frac{v_k}{m}  } }
\right)^q
e^{(\frac{u\cdot \lambda + v \cdot \gamma}{m} ) z}. 
\end{aligned}
\] 

Now, notice that if $\{ e_1,\dots ,e_p \}$ denotes the standard basis of $\C^p$, our selections of $(c_1,\dots ,c_p)$ ensure that
\begin{equation} \label{eq:uno}
\Phi^q(D) ((A+R)^m)-B = \sum_{\ell\in \mathcal{L}_m^*} U_{m,\ell},
\end{equation}
where \[
\mathcal{L}_m^*=\{ \ell =(u, v) \in \mathcal{L}_m: \ |u|\ne 0 \mbox{ or } \ v\notin \{ me_1,\dots ,me_p \}  \}. \]
Also, for each $1\le j \le m$ and  $\ell=(u, v)\in \mathcal{L}_j$ with $|v|<m$ we have 
\[ 
U_{j, \ell}\underset{q\to\infty}{\to} 0,
\] 
as 
our selections of $\Lambda$ and $\Gamma$ give by $\eqref{eq:L1.2}$ that $\frac{u\cdot \lambda + v\cdot \gamma}{m}\in\Omega$ and thus
\[
\left|\Phi(\frac{u\cdot \lambda + v\cdot \gamma}{m})\right| < 1 < r=|\Phi (\gamma_1)|=\dots = |\Phi(\gamma_p)|.
\]
Hence since each $\mathcal{L}_j$ is finite we have
\begin{equation}\label{eq:b}
\Phi(D)^q ((A+R_q)^j) \underset{q\to\infty}{\to } 0 \ \ \ (1\le j < m).
\end{equation}
Finally, recall that by Lemma~\ref{L:1} we have 
\[
\conv(\Gamma_r) \setminus \Gamma_r \subseteq \Phi^{-1} (r \D).
\]
Hence if $\ell=(u, v)\in \mathcal{L}_m^*$ with $|v|=m$ (so $\| v \|_\infty <m$ and $u=0$) we also have that $U_{m, \ell}\underset{q\to\infty}{\to} 0$, as
\[
\left|\Phi (\frac{u\cdot \lambda + v\cdot \gamma}{m} )\right| =\left | \Phi (\frac{ v\cdot \gamma}{m} )\right| < r = |\Phi(\gamma_1)|^\frac{v_1}{m} \dots |\Phi(\gamma_p )|^\frac{v_p}{m}.
\]
Thus
\[
\Phi^q(D) ((A+R_q)^m)\underset{q\to\infty}{\to} B,
\]
and $\eqref{eq:1}$ follows  by $\eqref{eq:a}$ and $\eqref{eq:b}$.
\end{proof}

\subsection{Some consequences of Theorem~\ref{T:1}}

Theorem~\ref{T:1} complements  \cite[Thm.~1]{bes_conejero_papathanasiou2016convolution} and gives an alternative proof to those of Shkarin \cite[Thm.\ 4.1]{shkarin2010on} and Bayart and Matheron
\cite[Thm.\ 8.26]{bayart_matheron2009dynamics} that $D$ supports a hypercyclic algebra.

\begin{corollary} \label{C:2}
Let $P(z)=(a_0+a_1 z^k)^n$ with $|a_0|\le 1$, $a_1\ne 0$, and $k,n\in\N$. Then $P(D)$ supports a hypercyclic algebra on $H(\C)$.
\end{corollary}
\begin{proof}
Notice first that $Q_1(z)=a_0+z^k$ satisfies the assumptions of Theorem~\ref{T:1}, and hence so does $Q_2(z)=a_0+a_1 z^k$, by Remark~\ref{R:2}. The conclusion now follows by a result due to Ansari \cite{Ansari} that the set of hypercyclic vectors for an operator $T$ coincides with the corresponding set of hypercyclic vectors for any given iterate $T^n$ $(n\in\N)$.
\end{proof}

We may also apply Theorem~\ref{T:1} to convolution operators that are not induced by polynomials. 
\begin{example} \label{E:cos(z)}
The operators $\cos(aD)$  and $\sin(aD)$ support a hypercyclic algebra on $H(\C )$ if $a\ne 0$. To see this, notice first that by Lemma~\ref{L:0} we may assume that $a=1$. For the first example, notice that $\Phi(z)=\cos(z)$ is of exponential type and
\[
 |\Phi(z)|^2=|\cos(z)|^2=\cos^2(x)+\sinh^2(y)  \ \ (z=x+iy,  x,y\in\R).
\]
So $\Gamma =\{ (x, f(x)): \ 0\le x \le \pi  \} \subset \Phi^{-1}(\partial\D)$ for the smooth function $f:[0, \pi]\to [0, \infty), \ f(x)=\sinh^{-1}(\sin(x))$, which is concave down since its second derivative $f''(x)=\frac{-2\sin(x)}{(1+\sin^2(x))^\frac{3}{2}}$ is negative on $(0,\pi)$. Now 
\[
\conv(\Gamma\cup\{ 0\})\setminus (\Gamma\cup\{ 0\})
\]
 is the region bounded by the graph of $f$ and the $x$-axis, on which $|\Phi|<1$, and $\cos(D)$ supports a hypercyclic algebra by Theorem~\ref{T:1}. The proof for $\sin(D)$ follows similarly by considering instead the subarc \[ \Gamma_0 := \left\{ \left(x-\frac{\pi}{2}, \sinh^{-1}(\sin(x))\right): \ 0\le x\le \pi \right\}\] of $\{ z\in\mathbb{C}: \ |\sin(z)|=1 \}$.
\end{example}

The next two examples should be contrasted with \cite[Corollary 2.4]{aron_conejero_peris_seoane-sepulveda2007powers}.
\begin{example}  \label{E:ze^z}
The operator $T=D\tau_1=De^D$ on $H(\C)$, where $\tau_1$ is the translation operator $g(z)\mapsto g(z+1), g\in H(\C)$ supports a hypercyclic algebra. 

Let $\Phi(z)=ze^z$. Clearly $\Phi$ is of exponential type, so we may check the conditions of Theorem~\ref{T:1}. Writing $z=x+iy$ we get
\begin{equation}
|f(z)|=1 \Leftrightarrow y^2=e^{-2x}-x^2
\end{equation}
The above equation has solutions provided the function  $\phi(x)=e^{-2x}-x^2$ satisfies that $\phi(x)\geq 0$. By doing some elementary calculus, one shows that $\phi$ is strictly decreasing on $\mathbb{R}$ and has a unique solution say $r \in (0,1)$. Thus the graph of the function
$$
h(x)=\sqrt{e^{-2x}-x^2},\quad x\in (-\infty ,r]
$$
lies in $f^{-1}(\partial \mathbb{D})$. Taking derivatives, we get that $h'<0$ and $h''<0$ on $(0,r)$, thus $h$ is strictly decreasing and concave down on $[0,r]$. Furthermore, it is evident that the sector 
$$
S=\{z=x+iy \in \C : 0\leq x<r, 0\leq y<h(x)\}
$$ 
lies in $f^{-1}(\mathbb{D})$. Thus, the strictly convex arc 
$$
\Gamma _1=\{z=x+iy \in \C : 0\leq x \leq r, y=h(x)\}$$ 
satisfies the conditions of Theorem~\ref{T:1}, which guarantees the existence of a hypercyclic algebra for the operator $f(D)$.
\end{example}

\begin{example} \label{E:e^z-a}
For each $0<a\leq 1$, the operator $T=\tau_1-aI=e^D-aI$ supports a hypercyclic algebra. To see this, we will show that the exponential type function $\Phi(z)=e^z-a$ satisfies the assumptions of Theorem~\ref{T:1}. If $z=x+iy$ then an easy calculation shows that
\begin{equation}\label{ex}
|\Phi(z)|\leq 1 \Leftrightarrow e^{2x}-2a\cos(y)e^x+a^2-1\leq 0.
\end{equation}
If we restrict $y\in [0,\frac{\pi}{4}]$ the inequality~$\eqref{ex}$ has solution $x\leq \log(a\cos(y)+\sqrt{1-a^2\sin^2(y)})$. Hence, setting 
$$
\Gamma _1=\left\{z=x+iy \in \C: 0\leq y\leq \frac{\pi}{4}, x=\log(a\cos(y)+\sqrt{1-a^2\sin^2(y)})\right\}
$$
we get that $\Gamma _1\subset \Phi^{-1}(\partial \D)$, and that 
$$
\{z=x+iy \in \C: 0\leq y\leq \frac{\pi}{4}, x< \log(a\cos{y}+\sqrt{1-a^2\sin^2(y)})\} \subset \Phi^{-1}(\D).
$$
Moreover, since $0<a\leq 1$ and $0\leq y \leq \frac{\pi}{4}$, it follows that 
$$
x=\log\left(a\cos(y)+\sqrt{1-a^2\sin^2(y)}\right)>0
$$
and that
\begin{align*}
&\frac{dx}{dy}=-\frac{a\sin(y)}{\sqrt{1-a^2\sin^2(y)}}<0,\\
&\frac{d^2x}{dy^2}=-\frac{a\cos(y)}{(1-a^2\sin^2(y))^{3/2}}<0.
\end{align*}
Hence, the function $x=\log\left(a\cos(y)+\sqrt{1-a^2\sin^2(y)}\right), y\in [0,\frac{\pi}{4}]$ is positive, decreasing and concave down. It follows that $\conv(\Gamma _1)\setminus \Gamma _1 \subset \Phi^{-1}(\D)$, and hence by Theorem~\ref{T:1} that $\Phi(D)$ has a hypercyclic algebra as claimed.
\end{example}

The following observation by Godefroy and Shapiro allows to conclude the existence of hypercyclic algebras for differentiation operators on $C^\infty (\R , \C)$. 
\begin{remark} (Godefroy and Shapiro)\label{R:GoS}
The restriction operator
\[
\begin{aligned}
\mathcal{R}:&H(\C)\to C^\infty(\R, \C) \\
f&=f(z)\mapsto f(x)
\end{aligned}
\]
is continuous, of dense range, and multiplicative, and for any complex polynomial $P=P(z)$ 
we have
\[
\mathcal{R} P\left(\frac{d}{dz}\right)  = P\left(\frac{d}{dx}\right) \mathcal{R} 
\]
\end{remark}
By Theorem~\ref{T:1},  Remark~\ref{R:2}, Remark~\ref{R:GoS} and \cite[Thm.~1]{bes_conejero_papathanasiou2016convolution}  we have: 
\begin{corollary} \label{C:10}
Let $P \in H(\C )$ be either a non-constant polynomial vanishing at zero or so that the level set $\{ z\in\mathbb{C}: \ |\Phi(z)|=1 \}$ contains a non-trivial, strictly convex compact arc $\Gamma_1$ satisfying
\[
\conv(\Gamma_1\cup \{ 0 \} ) \setminus ( \Gamma_1\cup \{ 0 \} ) \subseteq P^{-1} (\mathbb{D}).
\]
Then the operator $P(\frac{d}{dx})$ supports a hypercyclic algebra on $C^\infty (\R , \C )$. 
In particular, $T=aI+b\frac{d}{dx}$ supports a hypercyclic algebra on $C^\infty (\R , \C )$ whenever $|a|\le 1$ and $0\ne b$.
\end{corollary}

%

\begin{remark}
The geometric assumption $\eqref{eq:T1.1}$  in
Theorem~\ref{T:1} does not seem to be a necessary one, as the following example by F\'elix Mart\'{\i}nez suggests.
The polynomial  $P(z):= \frac{9^{9/8}}{8}z(z^8-1)$ vanishes at zero, so $P(D)$ supports a hypercyclic algebra by \cite[Thm.~1]{bes_conejero_papathanasiou2016convolution}. On the other hand, numerical evidence suggests that the level set $P^{-1}(\partial\D)$ 
 does not contain any non-trivial strictly convex compact arc $\Gamma$ 
so that $\conv(\Gamma \cup\{ 0 \}) \setminus (\Gamma \cup \{ 0 \}) \subset P^{-1} (\D)$, see 
Figure~\ref{fig:rose}. 

\begin{figure}[ht] 
\centering
\includegraphics[width=60mm]{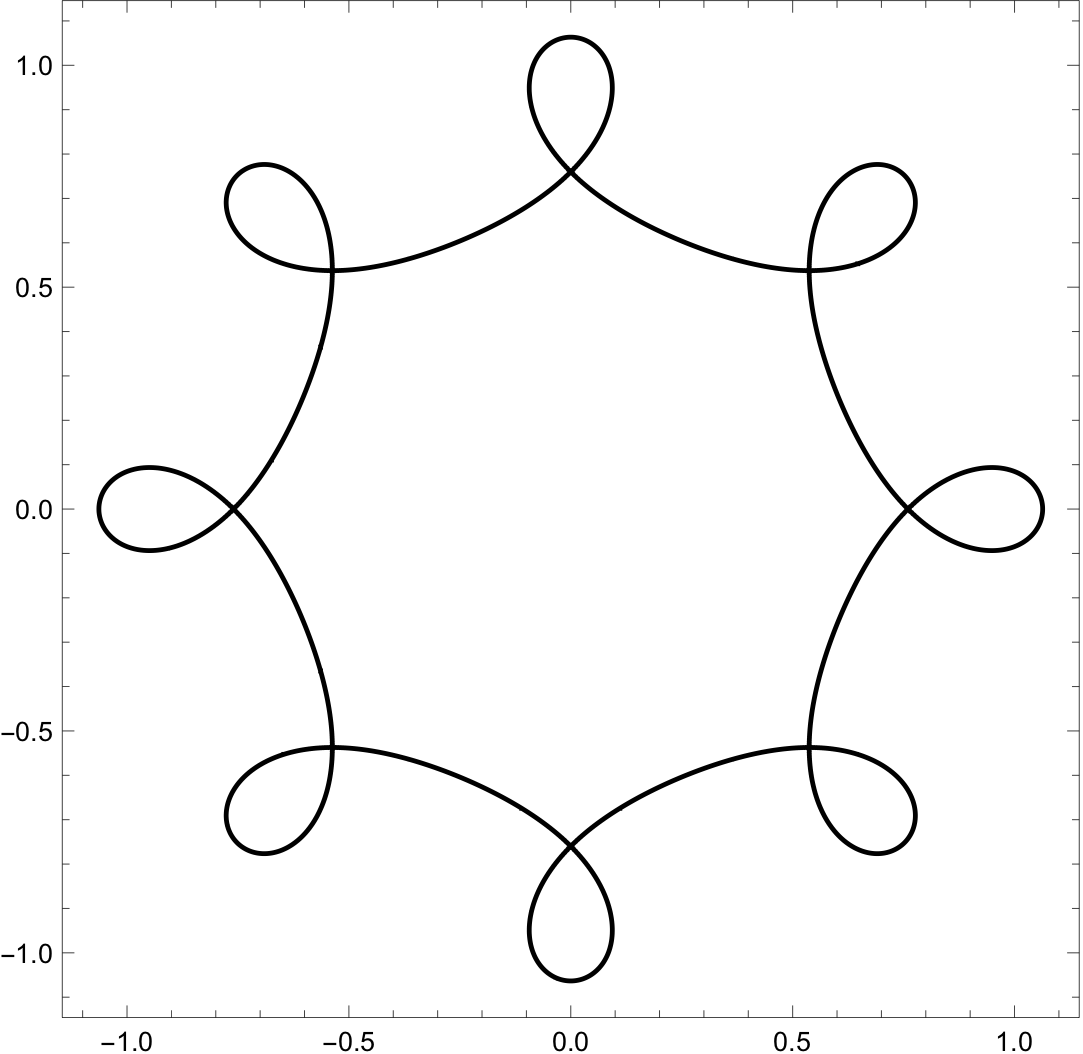}
\caption{The level set $P^{-1}(\partial\D)$  for $P(z):= \frac{9^{9/8}}{8}z(z^8-1)$.
}
\label{fig:rose}
\end{figure}
\end{remark}

We conclude the section by raising the following question, see Theorem~\ref{T:trio}.

\begin{question}\label{Pr:1}
Let $\Phi(D): H(\C)\to H(\C)$ be a hypercyclic convolution operator {\em not} supporting a hypercyclic algebra. 
\begin{enumerate}
\item[(i)]\ (Aron) Can $\Phi$ be a polynomial?
\item[(ii)]\ Must $\Phi\in H(\C )$ be of the form $\Phi(z)=b e^{a z}$, for some non-zero scalars $a$ and $b$? 
\end{enumerate}
\end{question}

\section{Hypercyclic algebras and composition operators} \label{S:CompositionOperators}

\subsection{Translations on the algebra of smooth functions}
As observed in Theorem~\ref{T:-1} by Aron et al \cite{aron_conejero_peris_seoane-sepulveda2007powers} no translation operator $\tau_a(f)(\cdot )=f(\cdot +a)$ supports a hypercyclic algebra on $H(\C )$. We show in Corollary~\ref{C:HcAsmoothtranslationsC} below that in contrast each non-trivial translation $\tau_a$ supports a hypercyclic algebra on $C^\infty(\R,\C)$. The proof is based on the following general fact.

\begin{theorem} \label{P:duo}
Let $T$ be a hypercyclic multiplicative operator on a separable $F$-algebra $X$ over the real or complex scalar field $\mathbb{K}$. The following are equivalent:
\begin{enumerate}
\item[{\rm ($a$)}]\ The operator $T$ supports a hypercyclic algebra.

\item[{\rm ($b$)}]\ For each non-constant polynomial $P\in \mathbb{K}[t]$ with $P(0)=0$, the map $\widehat{P} : X\to X$, \ $f\mapsto P(f)$, has dense range.

\item[{\rm($c$)}]\ Each hypercyclic vector for $T$ generates a hypercyclic algebra.
\end{enumerate}
\end{theorem}

\begin{proof}
The implication $(c)\Rightarrow (a)$ is immediate. To see $(a)\Rightarrow (b)$, let $g\in X$  generating a hypercyclic algebra $A(g)$ for $T$. In particular, for each polynomial $P$ vanishing at the origin the multiplicativity of $T$ gives that  $\widehat{P}(\mbox{Orb}(T, g))= \mbox{Orb}(T, P(g))$ is dense in $X$. Finally, to show the implication $(b)\Rightarrow (c)$, let $f\in X$ be hypercyclic for $T$ and let   $P\in \mathbb{K}[t]$ be a non-constant polynomial with $P(0)=0$. Given $V\subset X$ open and non-empty, 
by our assumption the set $\widehat{P}^{-1}(V)$ is open and non-empty. So there exists $n\in\N$ for which $T^n(f)\in \widehat{P}^{-1}(V)$.  The multiplicativity of the operator $T$ now gives
\[
T^n(P(f))=\widehat{P}(T^n(f))\in V.
\]
So $P(f)$ is hypercyclic for $T$ for each non-constant polynomial $P$ that vanishes at zero.
\end{proof}

\begin{corollary}\label{C:T_aR}
For each $0\ne a\in \R$ the translation operator \[ T_a: C^\infty(\R,\R)\to C^\infty(\R,\R), \ T_a(f)(x)=f(x+a), \ x\in\R,\] is weakly mixing  but does not support a hypercyclic algebra.
\end{corollary}
\begin{proof}
Notice first that $J\tau_a = (T_a\oplus T_a)J$ for the $\R$-linear homeomorphism $J:C^\infty(\R, \C)\to C^\infty(\R,\R)\times C^\infty(\R,\R)$, $J(f)=(\Re(f), \Im(f))$. So $T_a$ is weakly mixing. But the multiplicative operator $T_a$ does not support a hypercyclic algebra, since $\{ f^2: \ f\in C^\infty(\R,\R) \}$ is not dense in $C^\infty(\R, \R)$.
\end{proof}

The next two lemmas are used to establish Proposition~\ref{P:3}.

\begin{lemma}  \label{L:aux1}
Let $B\in C^{\infty}(\R,\C)$ be of compact support. Then for each countable subset $F$ of $\C$ and each  $\epsilon >0$ there exists $a\in D(0,\epsilon)$ such that \[ \Range(B)\cap (F-a)=\emptyset.\]
\end{lemma}
\begin{proof}
Let $N\in \N$ such that $B$ is supported in $[-N, N]$. Notice that the restriction of $B$ to $[-N,N]$ is a closed rectifiable planar curve. Now, suppose that $\Range(B)\cap (F-a) \neq \emptyset$
for each $a\in D(0,\epsilon)$. Then
\begin{align*}
D(0,\epsilon)=&\bigcup_{y\in F}\{a\in D(0,\epsilon): y-a \in\Range(B)\}\\
=& \bigcup_{y\in F}[D(0,\epsilon)\cap (y-\Range(B))],
\end{align*}
and since $F$ is countable we must have for some $y\in F$ that \[m(D(0,\epsilon)\cap (y-\Range(B)))>0,\] where $m$ denotes the two dimensional Lebesgue measure. But since $B$ is rectifiable 
$m(y-\Range(B))=m(\Range(B))=0$, a contradiction.
\end{proof}

\begin{lemma} \label{L:aux2}
Let $P$ be a polynomial of degree $m\in\N$ and let $B\in C^{\infty}(\R,\C)$ be constant outside of a compact set such that 
\[\Range(B)\cap P(\{P'=0\})=\emptyset.\] Then there exists $g\in C^{\infty}(\R,\C)$ such that $P(g)=B$.
\end{lemma}
\begin{proof}
Let $N\in \N$ such that $B$ is constant outside $[-N,N]$. Consider the family $\mathcal{M}$ of tuples $((-\infty, b), h)$ with $b\in (-\infty, \infty]$ and $h\in C^\infty ((-\infty, b), \C)$ satisfying that  $P\circ h= B$ on $(-\infty, b)$. Endow $\mathcal{M}$ with the partial order $\le$ given by
$((-\infty, b_1), h_1)\le ((-\infty, b_2), h_2)$ if and only if both $b_1\le b_2 \mbox{ and } h_1=h_2$ on $(-\infty, b_1)$.
Observe that $\mathcal{M}\neq \emptyset$ as by picking $c\in P^{-1}(B(-N))$ and letting $h:(-\infty,-N)\to\C$, $h(x)=c$, we have $((-\infty,-N),h)\in \mathcal{M}$. 
Also, any totally ordered subfamily $\{((-\infty,b_j), h_j)\}_{j\in J}$ of $\mathcal{M}$ has $((-\infty, b), h)\in \mathcal{M}$ as an upper bound, where $b=\sup_{j\in J}b_j$ and where  $h:(-\infty,b)\rightarrow \C$ is defined by $h(x)=h_j(x)$ for $x\in (-\infty,b_j)$.
 It follows by Zorn's Lemma that $\mathcal{M}$ contains some maximal element $((-\infty,b),g)$. We claim that $b=\infty$. 
Now, if $b<\infty$ then since $\Range(B)\cap P(\{P'=0\})=\emptyset$ there exist $m$ distinct points $z_1,\dots ,z_m$ such that $P(z_1)=\dots = P(z_m)= B(b)$. Furthermore, we may find a neighbourhood $W$ of $B(b)$ and pairwise disjoint open sets $U_1,\dots, U_m$  with $(z_1,\dots, z_m)\in U_1\times\dots\times U_m$ and biholomorphisms $g_j: W\rightarrow U_j$ such that $g_j\circ P (z)=z$ for $z\in U_j$ $(j=1,\dots,m)$. Now, pick an open interval  $(a, c)\subset B^{-1}(W)$ with   $b\in (a, c)$. Notice that $g((a, b))\subset U_j$ for some unique $j$. Define $h$ on $(-\infty,c)$ by $h(x)=g(x)$ if $x\in (-\infty,b)$ and $h(x)=g_j\circ B(x)$ if $x\in (a, c)$. Then $h$ is well defined and in $C^{\infty}((-\infty,c),\C)$ and $((-\infty,b),g)<((-\infty,c),h)$ contradicting the maximality of $((-\infty,b),g)$. So $b=\infty$ and  $P(g)=B$, concluding the proof.  
\end{proof}

We note that Lemma~\ref{L:aux2} also follows from an (albeit longer) compactness argument that does not require Zorn's Lemma.

\begin{proposition} \label{P:3}
Let $P$ be a non-constant polynomial with complex coeficients and which vanishes at zero. Then $\widehat{P}:C^\infty(\R,\C)\to C^\infty(\R,\C)$, $f\mapsto P(f)$, has dense range.
\end{proposition}
\begin{proof}
Given $U\subset C^{\infty}(\R,\C)$ open and non-empty, pick $B\in U$ with compact support. By Lemma~\ref{L:aux1} there exists $a\in\C$ such that  $B_1:=B+a \in U$ and  $\Range(B_1)\cap P(\{P'=0\})=\emptyset$. By Lemma~\ref{L:aux2} there exists some $g\in C^{\infty}(\R,\C)$ such that $P(g)=B_1$.
\end{proof}

\begin{corollary}  \label{C:HcAsmoothtranslationsC}
For each $0\ne a\in \R$ the translation operator $\tau_a$ supports a hypercyclic algebra on $C^\infty(\R, \C)$.
\end{corollary}

\begin{proof}
The hypercyclicity of $\tau_a$ on $C^\infty(\R, \C)$ follows from Birkhoff's theorem and the so-called Comparison Principle; notice that $\tau_a \mathcal{R} = \mathcal{R} \tau_a$ where the translation on the right hand side is acting on $H(\C)$. The conclusion now follows by Proposition~\ref{P:3} and Theorem~\ref{P:duo}.
\end{proof}

\subsection{Composition operators on $H(\Omega)$}
Recall that given a domain $\Omega$ in the complex plane, each $\omega \in H(\Omega )$ and $\varphi: \Omega \to \Omega$ holomorphic induce a weighted composition operator
\[
C_{\omega, \varphi}:H(\Omega)\to H(\Omega), f\mapsto \omega (f\circ \varphi).
\]
When $C_{\omega, \varphi}$ is supercyclic, the weight symbol $\omega$ must be zero-free and the compositional symbol $\varphi$ must be univalent and without fixed points. Moreover, these conditions are sufficient for the hypercyclicity of $C_{\omega, \varphi}$ when $\Omega$ is simply connected \cite[Proposition~2.1 and Theorem~3.1]{Bes_CAOT}.  We conclude the paper by noting the following extension of Theorem~\ref{T:-1}.

\begin{theorem} \label{T:trio}
Let $\Omega\subset \C$ be a domain. Then no 
weighted composition operator $C_{\omega, \varphi}:H(\Omega)\to H(\Omega)$ supports a supercyclic algebra. 
\end{theorem}

\begin{proof}
Given any $f\in H(\Omega)$ and $a\in\Omega$,  the polynomial $g(z)=(z-a)^3$ is not in the closure of 
\[
\C \mbox{Orb}(C_{\omega, \varphi}, f^2) = \mbox{span}(\{ f^2\}) \cup \left\{  \lambda \prod_{j=0}^{n-1} C^j_{\varphi}(\omega) \ C_{\varphi}^n(f^2) : \ n\in\N, \lambda\in\C \right\}.
\]

Indeed, if $(n_k)$ is a strictly increasing sequence of positive integers and $(\lambda_k)$  is a scalar sequence satisfying 
\[
\lambda_k  \left( \prod_{j=0}^{n_k-1} C_{\varphi}^j(\omega ) \right) \, C_{\varphi}^{n_k} (f^2)=  \lambda_k C_{\omega, \varphi}^{n_k}(f^2) \underset{k\to\infty}{\to} g
\]
then  by Hurwitz theorem \cite[page 231]{gamelin} there exists a disc $D(a, \delta)\subset \Omega$ centered at $a$ so that for each large $k$
\[
\lambda_k \,\left( \prod_{j=0}^{n_k-1} C_{\varphi}^j (\omega) \right) \, C_{\varphi}^{n_k} (f^2) =  \lambda_k \, \left( \prod_{j=0}^{n_k-1} C_{\varphi}^j(\omega)\right) \,   (C_{\varphi}^{n_k}(f) )^2
\]
has exactly three zeroes (counted with multiplicity) on $D(a, \delta)$ which is impossible since $\omega$ is zero-free. 
\end{proof}

When $\Omega \subset \C$ is simply connected and $C_\varphi$ is a hypercyclic composition operator on $H(\Omega)$, then any operator in the algebra of operators generated by $C_\varphi$ is also hypercyclic \cite[Theorem~1]{bes_racsam}. Hence it is natural to ask:
\begin{question}
Let $C_\varphi$ be a hypercyclic composition operator on $H(\Omega)$, where $\Omega$ is simply connected, and let $P$ be a non-constant polynomial with $P(0)=0$. Can $P(C_\varphi)$ support a hypercyclic algebra?
\end{question}
The answer must be affirmative if Question~\ref{Pr:1}(ii) has an affirmative answer.  
Finally, notice that in contrast with Theorem~\ref{T:trio},  by Corollary~\ref{C:HcAsmoothtranslationsC} it is possible for a composition operator to support a hypercyclic algebra on 
$C^\infty(\R, \C)$.
The hypercyclic weighted composition operators on $C^\infty(\Omega, \C)$, where $\Omega\subset \R^d$ is open,  have been characterized in \cite{Przestacki}, see also \cite{bonet}. We conclude the paper with the following question.

\begin{question}
Let $\Omega\subset \R^d$ be open and nonempty. Which weighted composition operators on $C^\infty(\Omega, \C)$ support a hypercyclic algebra?
\end{question}

\end{document}